\documentclass{article}
\usepackage{arxiv}

\usepackage[utf8]{inputenc}
\usepackage[T1]{fontenc}
\usepackage[numbers]{natbib}
\usepackage{hyperref}
\usepackage{url}
\usepackage{booktabs}
\usepackage{amsfonts}
\usepackage{amsmath}
\usepackage{amssymb}
\usepackage{amsthm}
\usepackage{microtype}
\usepackage{graphicx}
\usepackage{xcolor}
\usepackage{multirow}
\usepackage{tikz}
\usetikzlibrary{svg.path}

\newtheorem{theorem}{Theorem}[section]
\newtheorem{definition}[theorem]{Definition}
\newtheorem{corollary}[theorem]{Corollary}

\renewcommand{\shorttitle}{Tensor Decomposition of Key-Value Caches}
\renewcommand{\undertitle}{}

\definecolor{orcidlogocol}{HTML}{A6CE39}
\tikzset{orcidlogo/.pic={
  \fill[orcidlogocol] svg{M256,128c0,70.7-57.3,128-128,128C57.3,256,0,198.7,0,128C0,57.3,57.3,0,128,0C198.7,0,256,57.3,256,128z};
  \fill[white] svg{M86.3,186.2H70.9V79.1h15.4v48.4V186.2z}
    svg{M108.9,79.1h41.6c39.6,0,57,28.3,57,53.6c0,27.5-21.5,53.6-56.8,53.6h-41.8V79.1z M124.3,172.4h24.5c34.9,0,42.9-26.5,42.9-39.7c0-21.5-13.7-39.7-43.7-39.7h-23.7V172.4z};
  \fill[white] svg{M88.7,56.8c0,5.5-4.5,10.1-10.1,10.1c-5.6,0-10.1-4.6-10.1-10.1c0-5.6,4.5-10.1,10.1-10.1C84.2,46.7,88.7,51.3,88.7,56.8z};
}}
\newcommand{\orcidlink}[1]{\href{https://orcid.org/#1}{\,\raisebox{-0.15ex}{\resizebox{!}{0.85em}{\begin{tikzpicture}[yscale=-1]\pic{orcidlogo};\end{tikzpicture}}}}}

\title{Tensor Decomposition of Transformer Key-Value Caches:
Spectral Structure and Format Comparison}

\author{%
  Rahul Krishnan\,\orcidlink{0009-0000-4542-7989} \\
  Universit\"at Trier
  \And
  Volker Schulz\,\orcidlink{0000-0001-7665-130X} \\
  Fachbereich IV, Mathematik, Universit\"at Trier
}

\begin{document}
\maketitle

\begin{abstract}

The key-value (KV) cache of autoregressive transformers can be viewed
as a fourth-order tensor spanning attention heads, tokens, features,
and grouped layers.  We measure the singular-value spectra of all four
mode unfoldings on Mistral-7B-v0.3 and LLaMA-2-13B and compare four
standard tensor decompositions: Tucker, CP, tensor train, and
t-SVD, at matched storage.  The spectra partition the four axes into
two classes.  The token and feature modes carry low-rank structure,
particularly for keys.  The head and layer modes are nearly full-rank
and resist compression at any practical error level.  Among the four
decompositions, Tucker achieves the lowest reconstruction error at
every compression ratio from $2\times$ to $5\times$, because it can
leave the full-rank modes untouched.  Comparisons with two-dimensional
unfolding baselines show that the preferred representation differs
between keys and values: 2D methods achieve lower key error, while
four-way Tucker achieves lower value error at matched storage.  A
mode-pinning theorem certifies the full-rank preservation from the
measured spectra alone.  Two further spectral properties affect the compressible modes
without touching the full-rank ones: values reach a higher error floor than keys at every
ratio, and post-RoPE keys lose $41\%$--$64\%$ of their pre-RoPE
compressibility on both models.
\end{abstract}


\section{Introduction}
\label{sec:intro}

Autoregressive transformers~\cite{vaswani2017attention} generate output
tokens one at a time.  At each step and each layer, the attention value or output is a weighted combination
of all previous tokens' value vectors, with weights determined by inner
products between the current token's query vector and all previous key
vectors.  Computing this output therefore requires the key and value
projections of every previous token at every layer.  Rather than
recomputing them, inference implementations store them for
reuse~\cite{shazeer2019multiquery}.  This stored pair of key and value
matrices is called the \emph{key-value (KV) cache}.  It grows linearly
in the batch size, the context length, the number of layers, and the
number and dimension of the attention heads.  Once contexts are long or
batches are large, it is the cache that fills
device memory and limits throughput~\cite{kwon2023pagedattention}.
Low-rank approximations computed via truncated
SVD~\cite{kolda2009tensor} are used by existing
methods~\cite{chang2024palu,chang2025xkv} to compress this cache.

At each layer~$\ell$, the KV cache after~$T$ tokens is the pair
$(\mathcal{K}_\ell,\, \mathcal{V}_\ell)$ of third-order tensors
\[
  \mathcal{K}_\ell \in \mathbb{R}^{n_h \times T \times d_h}, \quad
  \mathcal{V}_\ell \in \mathbb{R}^{n_h \times T \times d_h}.
\]
where~$n_h$ is the number of key-value heads, $T$~is the input prompt sequence
length, and~$d_h$ is the per-head feature dimension.  We refer to the
three axes as the head, token, and feature modes.  Stacking adjacent
layers adds a fourth, the layer mode.  The multi-linear structure of
the KV cache has not been systematically studied.

Multi-linear algebra offers several standard tensor decomposition methods to exploit correlations across all of these axes at once~\citep{kolda2009tensor}: Tucker~\citep{tucker1966some,delathauwer2000hosvd}, CP~\citep{carroll1970cp,harshman1970parafac}, tensor train~\citep{oseledets2011tt}, and t-SVD~\citep{kilmer2011tsvd}.
These decompositions have a long record on the \emph{weights} of a neural network: convolution kernels have been factored in CP and Tucker form~\citep{lebedev2015cp,kim2016tucker} and fully connected layers stored as tensor trains~\citep{novikov2015tensorizing}, with the rank of every mode chosen by hand, by a global tolerance or by a per-mode estimator such as VBMF~\citep{kim2016tucker}, rather than allocated under a storage budget.
The KV cache differs from weight tensors as it is regenerated with every new prompt, and the token mode is typically much larger than head mode. Whether the standard decompositions remain effective on the KV Cache tensor has not been structurally studied.
We investigate the compressibility of each tensor mode of the KV cache and compares the reconstruction performance of the four standard tensor decompositions on it.  We focus on the multi-linear structure and reconstruction properties.

We ask that structural question directly.
Which modes of the KV tensor carry compressible structure, which do not, and what does the answer imply for the choice of decomposition?
We extract real caches from two production models, Mistral-7B-v0.3~\citep{jiang2023mistral}, a grouped-query attention~\citep{ainslie2023gqa} model, and LLaMA-2-13B~\citep{touvron2023llama2}, a multi-head attention~\citep{vaswani2017attention} model, read off the singular-value spectrum of every mode unfolding, and fit the four decompositions at matched storage across all layers and repeated input draws.
This yields a mode-count dichotomy: the token and feature modes carry compressible low-rank structure, whereas the head and grouped-layer modes are short and exhibit little spectral decay. We call the latter \emph{index-like} (Definition~\ref{def:index-like}).

\paragraph{Contributions}

\begin{enumerate}
    \item A systematic mode-wise spectral characterization of real KV tensors on two production models, together with a mode-$k$ tail-energy lower bound that relates the observed forced-rank errors to the corresponding singular-value tails (Sections~\ref{sec:spectral} and~\ref{sec:rL}). 
    \item  A matched-storage comparison of Tucker, CP, tensor train, and t-SVD, and of two-dimensional cache unfoldings against the four-mode Tucker (Sections~\ref{sec:formats} and~\ref{sec:unfoldings}), including a key/value split that is not visible in any single format.
    \item A formal statement of the mode-count dichotomy for the Lagrangian rank allocator (Theorem~\ref{thm:pinning} and Corollary~\ref{cor:tightness}, Section~\ref{sec:certificate}) that certifies mode pinning from the measured spectra alone.
    \item  A controlled study of the K/V asymmetry and the pre/post rotary-embedding penalty as spectral properties of the cache rather than of any single decomposition (Sections~\ref{sec:kv} and~\ref{sec:rope}).

\end{enumerate}

The rest of the paper is organized as follows.
Section~\ref{sec:background} fixes notation and recalls the four decompositions.
Section~\ref{sec:spectral} measures the per-mode singular-value spectra.
Section~\ref{sec:formats} compares the four formats at matched storage.
Section~\ref{sec:unfoldings} tests the small-count hypothesis against the two-dimensional unfoldings and against a free four-mode allocator with a scripted rank ablation.
Sections~\ref{sec:kv} and~\ref{sec:rope} isolate the K/V asymmetry and the pre/post rotary-embedding penalty.  Section~\ref{sec:conclusion} concludes.

\section{Background}
\label{sec:background}

\subsection{Autoregressive Transformers}
\label{sec:kv-cache-background}
\paragraph{The KV Cache}
At layer~$\ell$ and position~$t$ in a transformer model, the input
$x_t^{(\ell)} \in \mathbb{R}^d$ is projected into a query
$q_t^h = W_Q^{(\ell),h}\, x_t^{(\ell)}$, a key
$k_t^h = W_K^{(\ell),h}\, x_t^{(\ell)}$, and a value
$v_t^h = W_V^{(\ell),h}\, x_t^{(\ell)}$, all in~$\mathbb{R}^{d_h}$,
where $d_h$~the per-head feature dimension.  This projection is applied~$n_h$ times in parallel, each with
independent weight matrices. Each such copy is called an
\emph{attention head}, and~$h \in \{1,\ldots,n_h\}$ indexes it. The attention output at head~$h$ and
position~$t$ is
\begin{equation*}
  a_t^h \;=\; V_{1:t}^{h\!\top}\,
  \operatorname{softmax}\!\Big(
    K_{1:t}^{h}\, q_t^h \big/ \sqrt{d_h}
  \Big),
\end{equation*}
where~$K_{1:t}^h$ and~$V_{1:t}^h$ stack the keys and values of all
positions up to~$t$.  The query is computed from the current token
alone and is not cached.

We focus on two attention architectures: multi-head attention
(MHA)~\cite{vaswani2017attention} and grouped-query attention
(GQA)~\cite{ainslie2023gqa}.  Under MHA each query head has its own
key-value head.  Under GQA, multiple query heads share a single
key-value head, so the key-value head count~$n_h$ is smaller than the
query-head count. Throughout this paper, $n_h$~denotes the key-value
head count.

\paragraph{Rotary position embedding}
The attention equation above is permutation-invariant: it scores each
past key against the current query regardless of where that key
appeared in the sequence. To make the output sensitive to word order,
position information must be injected into the keys and queries.
\emph{Rotary Position Embedding} (RoPE)~\citep{su2021roformer} does
so by rotating each key and query vector by an angle that depends on
its position in the sequence. The $d_h$ components of each head are
grouped into $d_h/2$ coordinate pairs, and each pair is independently
rotated by a position-dependent angle, making the full operator $R_t$
block-diagonal and orthogonal. The values are not rotated. Since
$R_t$ depends on the position $t$, applying it changes the
singular-value spectrum of the key cache. \emph{Pre-RoPE} stores the
unrotated key $k_t^h$ and applies $R_t$ at read time.
\emph{Post-RoPE} stores the rotated key
$\tilde{k}_t^h = R_t\, k_t^h$ directly.

\subsection{Tensor Decomposition formats}

\paragraph{Notation}
For a tensor $\mathcal{Y} \in \mathbb{R}^{n_1 \times \cdots \times n_d}$,
the mode-$k$ unfolding $\mathcal{Y}_{(k)} \in \mathbb{R}^{n_k \times \prod_{j \neq k} n_j}$
is the matrix obtained by reshaping $\mathcal{Y}$ so that mode $k$
indexes the rows and all remaining modes are collapsed into the columns.
The mode-$k$ product $\mathcal{Y} \times_k M$ with
$M \in \mathbb{R}^{m \times n_k}$ is defined by
$(\mathcal{Y} \times_k M)_{(k)} = M\,\mathcal{Y}_{(k)}$.
Throughout, approximation quality is the relative Frobenius error
$\|\mathcal{Y}-\tilde{\mathcal{Y}}\|_F / \|\mathcal{Y}\|_F$, where $\tilde{\mathcal{Y}}$ is the reconstructed approximation,
averaged over all layers unless stated otherwise, and compression
is the ratio of uncompressed to stored scalar counts.

\paragraph{Tensor-decomposition algorithms} We compare four standard low-rank tensor formats, each based on a different notion of tensor rank~\citep{kolda2009tensor}.  For a third-order tensor $\mathcal X \in \mathbb R^{n_h \times T \times d_h}$ their approximations take the following forms.
\begin{itemize}
\item \textbf{Tucker} represents a tensor by a smaller core
$\mathcal{G} \in \mathbb{R}^{r_1 \times r_2 \times r_3}$ and one
factor matrix $U^{(k)} \in \mathbb{R}^{n_k \times r_k}$ per mode,
with independently chosen multilinear ranks $(r_1, r_2, r_3)$:
\[
  \mathcal{X} \;\approx\; \mathcal{G} \times_1 U^{(1)}
  \times_2 U^{(2)} \times_3 U^{(3)}
\]
The factors are initialized by the higher-order singular value
decomposition (HOSVD)~\citep{tucker1966some,delathauwer2000hosvd} and refined by alternating least squares
(HOOI)~\citep{delathauwer2000rank}.
Riemannian trust-region methods on the fixed-multi-linear-rank manifold
provide an alternative optimization
approach~\citep{heidel2018riemannian}.
\item \textbf{CANDECOMP/PARAFAC (CP)} writes a tensor as a sum of $R$ rank-one outer products and fits the factors by alternating least squares~\citep{carroll1970cp,harshman1970parafac}:
\[
  \mathcal X \;\approx\; \sum_{r=1}^{R} u_r^{(1)} \circ u_r^{(2)} \circ u_r^{(3)},
\]
where $\circ$ denotes the outer product.
\item \textbf{Tensor train (TT)} represents a tensor by a chain of third-order cores joined by bond ranks that control the dimensions between adjacent modes~\citep{oseledets2011tt}:
\[
  \mathcal X(i_1,i_2,i_3) \;\approx\; G_1(i_1)\, G_2(i_2)\, G_3(i_3),
  \quad G_k \in \mathbb R^{r_{k-1} \times n_k \times r_k},\; r_0 = r_3 = 1.
\]
\item \textbf{t-SVD} applies matrix SVDs to frontal slices in the discrete Fourier domain along the feature mode (mode $3$ under our ordering) and truncates the resulting tubal rank~\citep{kilmer2011tsvd,kilmer2013tubal}:
\[
  \mathcal X \;=\; \mathcal U * \mathcal S * \mathcal V^\top,
\]
with $*$ the tensor--tensor product in the DFT domain and $\mathcal S$ tubally diagonal with tubal rank $r$.
\end{itemize}

\subsection{Experimental Setup}
\label{sec:methodology}
\paragraph{Methodology}
All experiments use the fixed mode ordering
$(n_h,\, T,\, d_h)$ for three-mode tensors. When a fourth mode
is added, $L_g = 4$ adjacent layers are stacked into one group,
giving the ordering $(n_h,\, T,\, d_h,\, L_g)$. For t-SVD, the
feature mode serves as the transform mode.

\paragraph{Compression metric}  Storage is reported as a parameter-count ratio $C_{\mathrm{param}} = N_{\mathrm{original}}/N_{\mathrm{stored}}$ in all reconstruction sweeps.

\paragraph{Group-level error} When $L_g$ adjacent layers are
stacked into a group tensor $\mathcal{X}$, the layer slices
$X_1, \ldots, X_{L_g}$ occupy disjoint entries, so the squared
Frobenius norm is additive over the slice index.  The group-level
relative error is therefore the norm-weighted aggregation
\begin{equation*}
  E_{\mathrm{group}}
  \;=\;
  \sqrt{\,\sum_{\ell=1}^{L_g}
    \|X_\ell\|_F^2\, e_\ell^2
    \;\Big/\;
    \sum_{\ell=1}^{L_g} \|X_\ell\|_F^2\,},
\end{equation*}
where $e_\ell$ is the per-layer relative Frobenius error.  This
is the aggregation the Eckart--Young bound operates on, and it
reduces to an unweighted root-mean-square when all layers have
equal norm.

\paragraph{Storage functions}
For a third-order tensor of shape $n_h \times T \times d_h$,
the storage requirements of the four decompositions, measured
in scalar parameters, are
\begin{align*}
  S_{\mathrm{Tucker}}(r_1,r_2,r_3)
    &= r_1 r_2 r_3 + n_h r_1 + T r_2 + d_h r_3, \\
  S_{\mathrm{CP}}(R)
    &= R(n_h + T + d_h), \\
  S_{\mathrm{TT}}(r_1,r_2)
    &= n_h r_1 + r_1 T r_2 + r_2 d_h, \\
  S_{\mathrm{t\text{-}SVD}}(r)
    &= r(n_h + T)d_h.
\end{align*}

For real-valued input, conjugate symmetry halves the independent
Fourier slices, so the t-SVD storage reduces to $r(n_h + T)d_h$
real parameters.

The storage functions count the decomposition cores and
factor matrices. When a Tucker mode retains its full rank,
its factor can be chosen as the identity and need not be
stored; the corresponding term $n_h r_1$, $T r_2$, or
$d_h r_3$ is therefore omitted.

For the fourth-order tensor with $L_g$ adjacent layers,
Tucker storage extends to
\begin{equation*}
  S_{\mathrm{Tucker}}^{(4)}(r_1,r_2,r_3,r_L)
  = r_1 r_2 r_3 r_L
  + n_h r_1 + T r_2 + d_h r_3 + L_g r_L,
\end{equation*}
with the factor term of any full-rank mode omitted.
Residual storage and associated metadata are included
separately when computing the total storage in bytes.

\paragraph{Rank allocation}  Ranks are chosen per group and per tensor by a Lagrangian allocator that solves
\begin{align*}
  \min_{\mathbf r^K,\,\mathbf r^V} \quad & \|\mathcal{K} - \tilde{\mathcal{K}}(\mathbf r^K)\|_F^2 + \|\mathcal{V} - \tilde{\mathcal{V}}(\mathbf r^V)\|_F^2 \\
  \text{subject to} \quad & S(\mathbf r^K) + S(\mathbf r^V) \le B / C_{\mathrm{target}},
\end{align*}
where $B$ is the group's uncompressed parameter count and $C_{\mathrm{target}}$ the target compression ratio.  Keys and values therefore have independent rank vectors $\mathbf r^K, \mathbf r^V$.  The objective admits two budget accountings: \emph{per-tensor}, where each of $K$ and $V$ meets the storage target on its own, and \emph{joint}, where the pair meets it together and the allocator can trade key rank for value rank at matched joint storage.  The format sweep of Section~\ref{sec:formats} uses the per-tensor budget; the K/V-split experiment of Section~\ref{sec:unfoldings}, including the Palu-style and xKV methods, uses the joint budget.  Numbers under the two conventions are not directly comparable, and each table names the convention it uses.

\paragraph{Solvers}  Tucker is fitted by ST-HOSVD~\citep{vannieuwenhoven2012sthosvd}; unless stated otherwise, ten HOOI iterations follow, but they do not change the ST-HOSVD error on this data.  CP is fitted by alternating least squares~\citep{kossaifi2019tensorly} in Tensorly.  TT is fitted by TT-SVD~\citep{oseledets2011tt}; t-SVD is fitted by DFT-domain matrix SVDs.

\paragraph{Compute and inputs}  Caches are extracted in fp16 on an NVIDIA A100 (40\,GB); all decompositions are fitted in fp16.  Each reconstruction sweep aggregates three independent prompt draws and every layer of the model ($32$ for Mistral, $40$ for LLaMA); prompts are non-overlapping chunks of the WikiText-2 validation split.  Layers within a single prompt share text and are therefore not independent samples, which we take into account when interpreting aggregate statistics.

\section{Spectral Structure of KV Caches}
\label{sec:spectral}
The KV cache is deeply asymmetric in its mode sizes. In our
experiments $T = 1024$ and $d_h = 128$, both large, whereas
$n_h$ is only $8$ (Mistral, GQA) or $40$ (LLaMA-2-13B, MHA),
and $L_g = 4$. Since each per-head slice is a $T \times d_h$
matrix, its rank is bounded by $\min(T, d_h) = 128$, and
empirically every slice already uses all available directions;
the rank ceiling therefore says nothing about compressibility.
Any reduction must come from spectral decay within those
directions.

We measure where that decay lives. The cache of Mistral-7B
layer~$15$ is unfolded along each of its three axes, and the
singular values of every unfolding are computed. Keys are taken
pre-RoPE, a choice Section~\ref{sec:rope} justifies.
Table~\ref{tab:spectral} exhibits a clear \emph{mode-count dichotomy}: the head mode requires its
full rank at both error thresholds, while the token and feature
modes admit substantial rank reduction. To make this distinction
precise, we introduce the following definition.

\begin{definition}[Index-like mode]\label{def:index-like}
Let $\mathcal{X} \in \mathbb{R}^{n_1 \times \cdots \times n_d}$ be a real order-$d$ tensor, $k \in \{1, \ldots, d\}$ a mode of $\mathcal{X}$ with dimension $n_k$, and $\varepsilon > 0$ an error threshold.  Let $X_{(k)} \in \mathbb{R}^{n_k \times \prod_{j \ne k} n_j}$ denote the mode-$k$ unfolding of $\mathcal{X}$ and $\sigma_1(X_{(k)}) \ge \cdots \ge \sigma_{n_k}(X_{(k)})$ its singular values.  The \emph{normalised mode-$k$ tail energy at rank $r_k$} is
\begin{equation*}
  L_k(r_k) \;=\; \frac{\bigg(\sum_{j > r_k} \sigma_j^2(X_{(k)})\bigg)^{1/2}}{\|\mathcal{X}\|_F}.
\end{equation*}
We call mode $k$ \emph{index-like at level $\varepsilon$} for $\mathcal{X}$ if
\begin{equation*}
  L_k(n_k - 1) \;>\; \varepsilon,
\end{equation*}
that is, if the smallest singular value $\sigma_{n_k}(X_{(k)})$ of the unfolding alone already carries more than $\varepsilon^2$ share of the squared energy $\|\mathcal{X}\|_F^2$.  No low-rank approximation whose mode-$k$ rank is strictly below $n_k$ can then achieve relative Frobenius error $\varepsilon$.
\end{definition}

\begin{table}[htbp]
  \centering
  \small
  \caption{Per-mode spectral properties of the Mistral-7B
    layer-$15$ cache ($T=1024$), from the mode-$k$ SVD of each
    unfolding of the pre-RoPE key tensor $\mathcal{K}$ and the
    value tensor $\mathcal{V}$. Ranges are min-to-max over $5$
    independent draws. $L_k(n_k{-}1)$ is the tail energy from
    Definition~\ref{def:index-like}: values above $\varepsilon$
    indicate an index-like mode at level $\varepsilon$.}
  \label{tab:spectral}
  \begin{tabular}{llrrrr}
    \toprule
    Tensor & Mode & $\sigma_1/\sigma_{\min}$
      & $r$ for $\le 10\%$ & $r$ for $\le 20\%$
      & $L_k(n_k{-}1)$ \\
    \midrule
    K & heads ($8$)       & $1.4$
      & $8/8$      & $8/8$      & $0.275$ -- $0.298$ \\
    K & tokens ($1024$)   & $0.4$ -- $2.9$\,M
      & $225/1024$ & $89/1024$  & $3{\times}10^{-7}$ -- $2{\times}10^{-6}$ \\
    K & features ($128$)  & $28.2$
      & $100/128$  & $69/128$   & $0.012$ -- $0.013$ \\
    \midrule
    V & heads ($8$)       & $1.4$
      & $8/8$      & $8/8$      & $0.278$ -- $0.296$ \\
    V & tokens ($1024$)   & $13$ -- $550$\,K
      & $576/1024$ & $382/1024$ & $5{\times}10^{-7}$ -- $2{\times}10^{-5}$ \\
    V & features ($128$)  & $2.7$
      & $126/128$  & $118/128$  & $0.056$ -- $0.059$ \\
    \bottomrule
  \end{tabular}
\end{table}

With this terminology, the three modes partition as follows.

\begin{itemize}
\item \emph{Head mode ($n_h = 8$), index-like.}
  $L_k(7) \approx 0.28$ on both keys and values, well above
  any practical error threshold. Every head carries comparable
  energy, and no rank reduction is possible without at least
  $28\%$ relative error.

\item \emph{Token mode ($T = 1024$), compressible.}
  $L_k(1023)$ is below $10^{-5}$, so the last singular
  direction contributes negligible energy. Keys are
  substantially more compressible than values: the rank needed
  for $10\%$ error is $225$ for keys versus $576$ for values.
  Figure~\ref{fig:spectral-decay} (top row) shows the key
  spectrum dropping by several orders of magnitude, while the
  value spectrum decays more gradually.

\item \emph{Feature mode ($d_h = 128$), compressible on keys,
  index-like on values.}
  On keys, $L_k(127) \approx 0.012$, leaving room for
  substantial rank reduction. On values, $L_k(127) \approx 0.057$,
  and the rank needed for $10\%$ error is $126$ of $128$, so
  feature-mode compression is effective only for keys.
  Figure~\ref{fig:spectral-decay} (bottom row) confirms the
  contrast: the key spectrum falls steeply, whereas the value
  spectrum remains nearly flat.
\end{itemize}

\begin{figure}[htbp]
  \centering
  \includegraphics[width=0.85\textwidth]{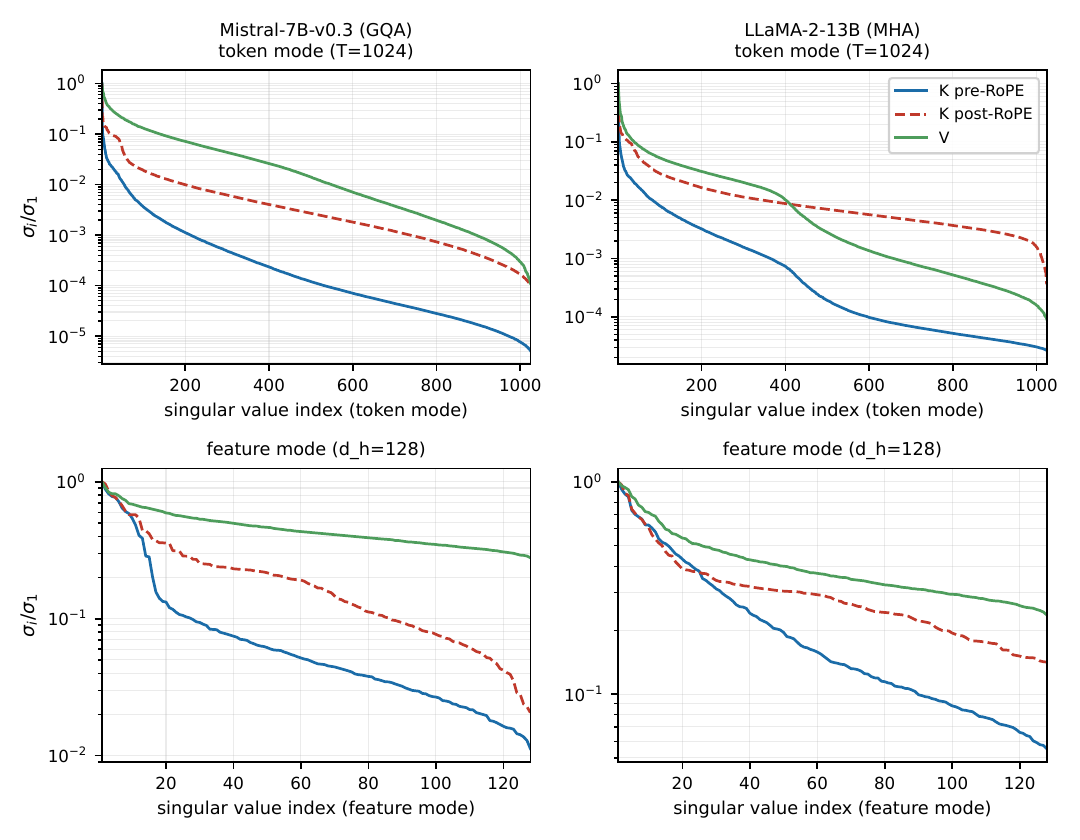}
  \caption{Normalized singular-value spectra $\sigma_i / \sigma_1$ of the token-mode (top row) and feature-mode (bottom row) unfoldings on layers 0--1 of Mistral-7B-v0.3 (GQA, left) and LLaMA-2-13B (MHA, right), $T=1024$; curves for pre-RoPE keys, post-RoPE keys, and values; vertical axis logarithmic.}
  \label{fig:spectral-decay}
\end{figure}

\section{Tensor-Format Comparison}
\label{sec:formats}

We now proceed to determine which of the four tensor decomposition formats is best suited to exploit the spectral
structure identified above. Tucker, CP, TT, and t-SVD are fitted
to the pre-RoPE keys and values of every layer at $T=1024$ over
three prompts, and the mean relative Frobenius error is
reported at each compression ratio. The experiment is conducted on Mistral 7B and Llama 2 13B models with the results reported in 
Tables~\ref{tab:formats-mistral} and~\ref{tab:formats-llama} respectively. CP is omitted on LLaMA because matched-storage
fitting requires CP rank $R \approx 2199$, which exceeds the
memory of the device used for experimentation.

\begin{table}[htbp]
  \centering
  \caption{Format comparison on Mistral-7B-v0.3 (GQA, $n_h=8$),
    \emph{per-tensor storage budget}: mean relative Frobenius
    error over all $32$ layers and $3$ prompts at $T=1024$.
    Lowest K and V errors per ratio in bold.}
  \label{tab:formats-mistral}
  \begin{tabular}{lrrrrrrrr}
    \toprule
    & \multicolumn{2}{c}{Tucker} & \multicolumn{2}{c}{CP}
    & \multicolumn{2}{c}{t-SVD} & \multicolumn{2}{c}{TT} \\
    Ratio & K & V & K & V & K & V & K & V \\
    \midrule
    $2\times$ & \textbf{0.087} & \textbf{0.279} & $0.098$ & $0.301$ & $0.147$ & $0.427$ & $0.224$ & $0.552$ \\
    $3\times$ & \textbf{0.125} & \textbf{0.381} & $0.155$ & $0.436$ & $0.204$ & $0.543$ & $0.303$ & $0.674$ \\
    $4\times$ & \textbf{0.153} & \textbf{0.451} & $0.194$ & $0.518$ & $0.243$ & $0.615$ & $0.349$ & $0.734$ \\
    $5\times$ & \textbf{0.176} & \textbf{0.503} & $0.224$ & $0.574$ & $0.268$ & $0.655$ & $0.387$ & $0.779$ \\
    \bottomrule
  \end{tabular}
\end{table}

\begin{table}[htbp]
  \centering
  \caption{Format comparison on LLaMA-2-13B (MHA, $n_h=40$),
    \emph{per-tensor storage budget}: mean relative Frobenius
    error over all $40$ layers and $3$ prompt draws at $T=1024$.
    CP omitted (see text). Lowest K and V errors per ratio in
    bold.}
  \label{tab:formats-llama}
  \begin{tabular}{lrrrrrr}
    \toprule
    & \multicolumn{2}{c}{Tucker} & \multicolumn{2}{c}{t-SVD}
    & \multicolumn{2}{c}{TT} \\
    Ratio & K & V & K & V & K & V \\
    \midrule
    $2\times$ & \textbf{0.089} & \textbf{0.232} & $0.167$ & $0.426$ & $0.278$ & $0.608$ \\
    $3\times$ & \textbf{0.127} & \textbf{0.328} & $0.222$ & $0.536$ & $0.373$ & $0.729$ \\
    $4\times$ & \textbf{0.153} & \textbf{0.392} & $0.258$ & $0.604$ & $0.434$ & $0.785$ \\
    $5\times$ & \textbf{0.173} & \textbf{0.437} & $0.279$ & $0.641$ & $0.495$ & $0.824$ \\
    \bottomrule
  \end{tabular}
\end{table}

From the tables, it is clear that Tucker achieves the lowest reconstruction error at every
compression ratio on both models. On Mistral the ordering is
Tucker $<$ CP $<$ t-SVD $<$ TT. On LLaMA it is
Tucker $<$ t-SVD $<$ TT. This ordering holds across all
individual layers and across seeds combination, not just in the aggregate.
Additionally, on this data when the Tucker allocator is free to choose all three mode
ranks, it selects the head and feature ranks at their full
counts ($r_1 = n_h = 8$, $r_3 = d_h = 128$) and truncates
only the token mode.

Both tables also show that value errors are consistently higher
than key errors across all formats, an asymmetry we further examine in
Section~\ref{sec:kv}.

The ordering reflects how each format handles the index-like
head mode. Tucker can leave it untouched and compress only the
token and feature axes. t-SVD applies a truncated SVD to
$n_h \times T$ frontal slices at every frequency, so head-mode
structure enters the truncation at each slice. TT threads a
chain through all three modes and pays a bond rank at each
link. Figure~\ref{fig:format-pareto} plots the error against
compression ratio for all formats on both models. Tucker
dominates across the full range, and the gap between formats
is larger on values than on keys, consistent with the harder
compressibility of the value spectrum noted in
Section~\ref{sec:spectral}.

\begin{figure}[htbp]
  \centering
  \includegraphics[width=0.9\textwidth]{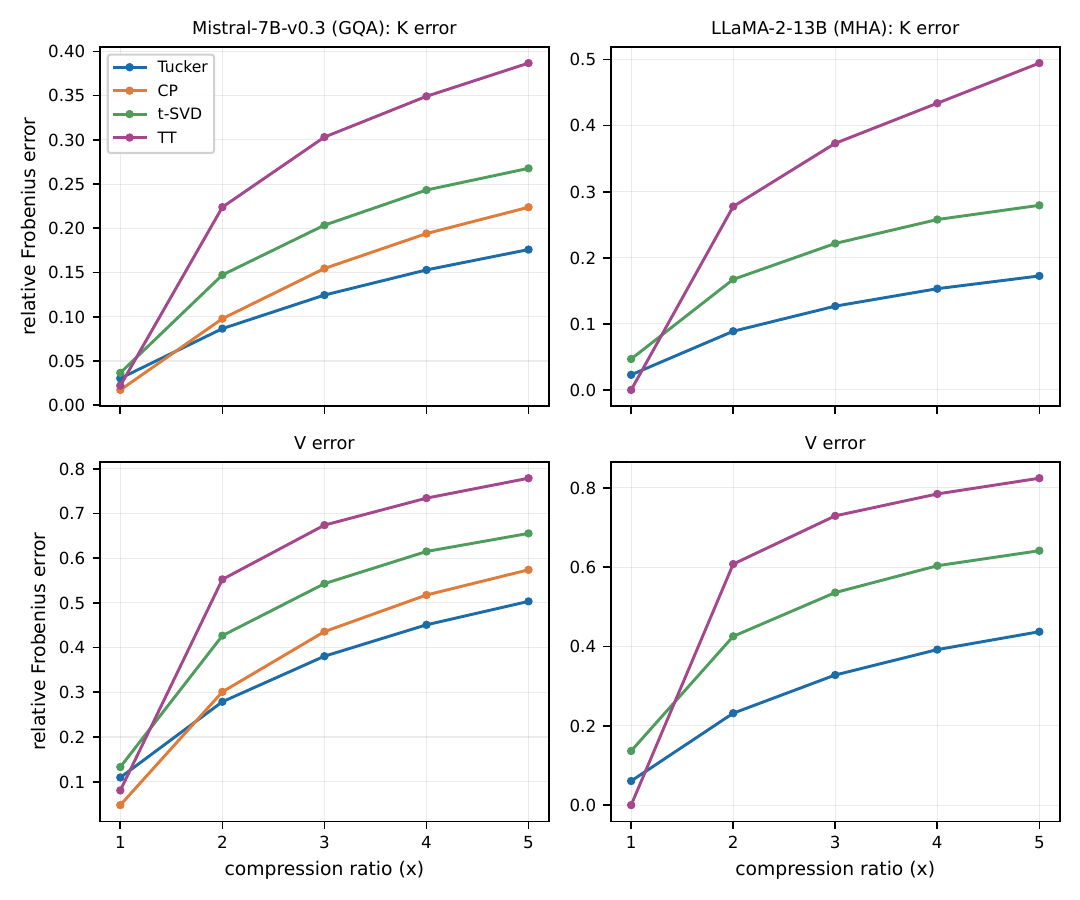}
  \caption{Reconstruction error against compression ratio for
    all evaluated formats on Mistral-7B-v0.3 and LLaMA-2-13B,
    keys and values, $T=1024$. Curves are the mean over all
    layers and three prompt draws.}
  \label{fig:format-pareto}
\end{figure}

\section{Higher-Order versus 2D Unfoldings}
\label{sec:unfoldings}

Among the evaluated third-order configurations, Tucker gives the lowest error over $2\times$--$5\times$ (Section~\ref{sec:formats}).  We next compare it against two-dimensional factorizations to test whether treating the KV cache as a third-order tensor improves reconstruction. The answer splits on the K/V boundary. On keys, the two-dimensional unfolding has lower error at every ratio tested. On values, full Tucker does. The split holds on both models.

\subsection{Per-head SVD}
\label{sec:perhead}

The simplest two-dimensional baseline is per-head SVD.  It treats each $(T, d_h)$ head slice as an independent matrix and truncates its singular values.  We compare it against Tucker with the head mode pinned at full rank ($r_1 = n_h$) and the token and feature ranks chosen by the Lagrangian allocator.  Both methods use matched storage, averaged over all $32$ Mistral layers and three prompt draws at $T = 1024$.  Table~\ref{tab:perhead} reports the results.

\begin{table}[htbp]
  \centering
  \caption{Tucker versus per-head SVD on Mistral-7B-v0.3 (GQA, $n_h=8$), all $32$ layers at $T=1024$, matched storage, $3$-seed mean.}
  \label{tab:perhead}
  \begin{tabular}{lrrrr}
    \toprule
    Ratio & K Tucker & K SVD & V Tucker & V SVD \\
    \midrule
    $2\times$ & $0.140$ & \textbf{0.139} & \textbf{0.155} & $0.421$ \\
    $3\times$ & \textbf{0.184} & $0.199$ & \textbf{0.251} & $0.534$ \\
    $4\times$ & \textbf{0.207} & $0.242$ & \textbf{0.325} & $0.604$ \\
    \bottomrule
  \end{tabular}
\end{table}

On values, Tucker has lower error at every ratio by a wide margin, from $0.155$ against $0.421$ at $2\times$ to $0.325$ against $0.604$ at $4\times$.  On keys, Tucker wins at $3\times$ and $4\times$.  At $2\times$ the two are within $0.002$, with per-head SVD marginally lower.

The gain does not come from compressing the head axis, which is pinned at full rank throughout.  It comes from shared bases.  Tucker fits a single token basis and a single feature basis shared across all heads,
so the factor matrices are stored once rather than once per head.  Per-head SVD fits
each slice independently and pays the full factor storage cost per head.  At matched storage,
this leaves Tucker with a larger rank budget in the token and feature modes.

\subsection{Cross-layer comparison: the K/V split}
\label{sec:kvsplit}
We now compare Tucker against two cross-layer two-dimensional
factorizations: the grouped low-rank decomposition of
Palu~\citep{chang2024palu} and the stacked-layer SVD of
xKV~\citep{chang2025xkv}.  The Palu-style baseline follows Palu's grouping of $4$ adjacent heads and unfolds the head mode into the feature mode to form a $(T,\, 4 \cdot d_h)$ matrix, but takes a truncated SVD of the cache itself, whereas Palu factors the key and value projection weights offline and caches the resulting latent. 
A single truncated SVD per group is taken at a rank
solved from the target compression ratio.  xKV unfolds the KV caches of $L_g = 4$ adjacent layers into
a single $(T,\, L_g \cdot n_h \cdot d_h)$ matrix and truncates
its SVD.  Tucker is tested in two variants: Tucker-3D takes one
layer at a time, and Tucker-4D stacks the same $L_g = 4$ layers
along a fourth mode, matching the xKV and PALU grouping for a controlled comparison, with ranks $(r_1, r_2, r_3, r_L)$.   All four
methods use a joint K/V storage budget, pre-RoPE keys,
and are averaged over three prompt draws at $T = 1024$.

\begin{table}[htbp]
  \centering
  \small
  \caption{Cross-layer K/V split, \emph{joint K/V storage budget}:
    Tucker (3D and 4D, fixed groups of $L_g=4$ layers) versus
    Palu and xKV,
    $T=1024$, $3$-seed mean.  }
  \label{tab:kvsplit}
  \footnotesize
  \begin{tabular}{llrrrrrrrr}
    \toprule
    & & \multicolumn{2}{c}{Palu} & \multicolumn{2}{c}{xKV}
      & \multicolumn{2}{c}{Tucker-3D} & \multicolumn{2}{c}{Tucker-4D} \\
    \cmidrule(lr){3-4}\cmidrule(lr){5-6}\cmidrule(lr){7-8}\cmidrule(lr){9-10}
    Model & Ratio & K & V & K & V & K & V & K & V \\
    \midrule
    Mistral-7B & $2\times$ & \textbf{0.096} & 0.312
      & 0.099 & 0.202 & 0.147 & 0.179
      & 0.134 & \textbf{0.143} \\
    & $3\times$ & 0.138 & 0.417
      & \textbf{0.133} & 0.301 & 0.190 & 0.282
      & 0.179 & \textbf{0.236} \\
    & $4\times$ & 0.168 & 0.486
      & \textbf{0.158} & 0.369 & 0.217 & 0.357
      & 0.201 & \textbf{0.308} \\
    \midrule
    LLaMA-2-13B & $2\times$ & 0.119 & 0.325
      & \textbf{0.113} & 0.186 & 0.142 & 0.138
      & 0.145 & \textbf{0.125} \\
    & $3\times$ & 0.162 & 0.423
      & \textbf{0.151} & 0.291 & 0.192 & 0.227
      & 0.196 & \textbf{0.218} \\
    & $4\times$ & 0.192 & 0.487
      & \textbf{0.176} & 0.359 & 0.215 & 0.297
      & 0.220 & \textbf{0.290} \\
    \bottomrule
  \end{tabular}
\end{table}
Table~\ref{tab:kvsplit} report the
results on Mistral ($32$ layers) and LLaMA ($40$ layers). The result is a clean K/V split on both models.  On keys, the 2D
methods win at every ratio, with xKV lowest from $3\times$ onward.
On values, Tucker-4D wins at every ratio, improving on Tucker-3D
throughout.  The split holds on LLaMA with the same structure.

The split is consistent with the spectra of
Section~\ref{sec:spectral}.  The key tensor has strong spectral
decay on both large-count modes, so a 2D factorization that
concentrates its entire budget on those modes captures the key
structure efficiently.  The value tensor has a nearly flat feature
spectrum, so 2D factorizations hit a higher floor.  Tucker avoids
this by sharing a compact basis across heads and assigning
independent ranks per mode.

\subsection{Head- and Layer-Rank Ablation}
\label{sec:rL}

In the K/V-split experiment, the four-mode allocator drove the
head and layer ranks to their full counts ($r_1 = n_h$,
$r_L = L_g$) on every group and ratio tested.  The question is
whether that choice is necessary.  We test this by forcing one
of the two small-count ranks below its full value and letting
the allocator redistribute the freed budget across the token
and feature modes at matched total storage.  Both models are
evaluated over three prompt draws at $2\times$ to $4\times$
in the same fixed-group setting ($L_g = 4$).

\begin{table}[htbp]
  \centering
  \caption{Head- and layer-rank ablation for Tucker-4D,
    $L_g=4$ fixed groups, $T=1024$, pre-RoPE, residual-free,
    $3$-seed mean.}
  \label{tab:mode-rank-ablation}
  \footnotesize
  \begin{tabular}{llrrrrrr}
    \toprule
    & & \multicolumn{2}{c}{$2\times$} & \multicolumn{2}{c}{$3\times$} & \multicolumn{2}{c}{$4\times$} \\
    \cmidrule(lr){3-4}\cmidrule(lr){5-6}\cmidrule(lr){7-8}
    Model & Arm & K & V & K & V & K & V \\
    \midrule
    \multirow{5}{*}{\shortstack[l]{Mistral\\($n_h=8$)}}
      & free ($r_h=8$, $r_L=4$) & $0.134$ & $0.143$ & $0.179$ & $0.236$ & $0.201$ & $0.308$ \\
      & $r_L=3$        & $0.393$ & $0.318$ & $0.417$ & $0.380$ & $0.432$ & $0.430$ \\
      & $r_L=2$        & $0.611$ & $0.524$ & $0.620$ & $0.553$ & $0.628$ & $0.581$ \\
      & $r_h=n_h/2=4$  & $0.667$ & $0.677$ & $0.670$ & $0.681$ & $0.674$ & $0.687$ \\
      & $r_h=n_h/4=2$  & $0.839$ & $0.847$ & $0.839$ & $0.847$ & $0.840$ & $0.847$ \\
    \midrule
    \multirow{5}{*}{\shortstack[l]{LLaMA\\($n_h=40$)}}
      & free ($r_h=40$, $r_L=4$) & $0.145$ & $0.125$ & $0.196$ & $0.218$ & $0.220$ & $0.290$ \\
      & $r_L=3$        & $0.417$ & $0.302$ & $0.441$ & $0.368$ & $0.457$ & $0.417$ \\
      & $r_L=2$        & $0.651$ & $0.529$ & $0.658$ & $0.550$ & $0.665$ & $0.578$ \\
      & $r_h=n_h/2=20$ & $0.594$ & $0.674$ & $0.598$ & $0.675$ & $0.603$ & $0.680$ \\
      & $r_h=n_h/4=10$ & --    & --    & --    & --    & $0.788$ & $0.843$ \\
    \bottomrule
  \end{tabular}
\end{table}

Table~\ref{tab:mode-rank-ablation} reports the results.
Truncating either the head rank or the layer rank below its full
count raises the reconstruction error sharply on both keys and
values, at every ratio, and on both models.  Even the mildest
constraint ($r_L = 3$, one rank below full) nearly triples the
key error on both models at $2\times$.  Halving the head rank
produces a comparable jump. On LLaMA, the $r_1 = n_h/4$ pin alone exceeds $3.3\times$
storage, so its $2\times$ and $3\times$ cells are not
matched-storage comparisons and are omitted from the table. A second pattern is visible across
the columns: the forced methods barely move as the compression
ratio relaxes from $2\times$ to $4\times$.  On Mistral, forcing
$r_1 = n_h/4$ gives key errors of $0.839$, $0.839$, and $0.840$
at the three ratios. A similar pattern can be observed for Llama as well. This implies that the forced rank rather than the storage
budget sets the error floor, and additional budget cannot
compensate for it. 

These error floors have a spectral explanation.  For any
approximation $\tilde{\mathcal{X}}$ whose mode-$k$ rank is at
most $r_k$, the Eckart--Young--Mirsky
theorem~\citep{eckart1936approximation,mirsky1960symmetric}
applied to the mode-$k$ unfolding gives
\begin{equation}
  L_k(r_k)\;=\;
  \frac{\left(\sum_{j > r_k}
    \sigma_j^2(X_{(k)})\right)^{1/2}}
  {\|\mathcal{X}\|_F}
  \;\le\;
  \frac{\|\mathcal{X} - \tilde{\mathcal{X}}\|_F}
  {\|\mathcal{X}\|_F},
  \label{eq:tail-bound}
\end{equation}
where $\sigma_j(X_{(k)})$ are the singular values of the
mode-$k$ unfolding.  The bound holds because unfolding preserves
the Frobenius norm, so
$\|\mathcal{X} - \tilde{\mathcal{X}}\|_F
= \|X_{(k)} - \tilde{X}_{(k)}\|_F$, and $\tilde{X}_{(k)}$ has
rank at most $r_k$.  When a mode has a flat spectrum, each additional rank captures
only a small share of the remaining energy, so $L_k(r_k)$
remains high until $r_k$ approaches $n_k$.  Since the bound
is format-agnostic, no algorithm can avoid this floor at a
given rank. 

\begin{table}[htbp]
  \centering
  \caption{Mode-$k$ tail-energy lower bound $L_k(r_k)$ versus
    measured Tucker-4D group-level error $E_{\mathrm{group}}$,
    Mistral, $T=1024$, target $2\times$, $3$-seed mean over
    $8$ groups.}
  \label{tab:bound-vs-measured}
  \begin{tabular}{llrrrrrr}
    \toprule
    Forced mode & $r_k$
      & \multicolumn{3}{c}{Keys}
      & \multicolumn{3}{c}{Values} \\
    \cmidrule(lr){3-5}\cmidrule(lr){6-8}
     & & $L_k(r_k)$ & $E_{\mathrm{group}}$ & Ratio
       & $L_k(r_k)$ & $E_{\mathrm{group}}$ & Ratio \\
    \midrule
    layer & $3$ & $0.460$ & $0.469$ & $1.02$ & $0.410$ & $0.420$ & $1.03$ \\
    layer & $2$ & $0.676$ & $0.677$ & $1.00$ & $0.621$ & $0.621$ & $1.00$ \\
    head  & $4$ & $0.666$ & $0.668$ & $1.00$ & $0.672$ & $0.672$ & $1.00$ \\
    head  & $2$ & $0.839$ & $0.839$ & $1.00$ & $0.843$ & $0.843$ & $1.00$ \\
    \bottomrule
  \end{tabular}
\end{table}

To test whether the ablation floors are near-optimal, we compare
the tail-energy bound $L_k(r_k)$ from~\eqref{eq:tail-bound} to
the measured group-level error $E_{\mathrm{group}}$ (defined in
Section~\ref{sec:methodology}) on Mistral at $2\times$.
Table~\ref{tab:bound-vs-measured} reports the comparison.  The
ratio $E_{\mathrm{group}}/L_k$ lies between $1.00$ and $1.03$
in every row, so the Tucker fit sits within a few percent of
the Eckart--Young lower bound of the truncated mode.
The bound is tightest when the forced rank is lowest ($r_k = 2$),
where the tail energy is large enough to dominate the total error
and the allocator has little room to affect the outcome. The ablation errors
therefore reflect the flat spectra of the head and layer modes
measured in Section~\ref{sec:spectral}, not a Tucker-specific
artifact.

\subsection{A spectral certificate for full-rank preservation}
\label{sec:certificate}

The free allocator in Section~\ref{sec:kvsplit} preserves the
head and grouped-layer ranks, $r_1=n_h$ and $r_L=L_g$, across
all groups and target ratios.  We now derive a sufficient
condition, checkable from the singular values of the mode-$k$ unfoldings referred to as mode spectra
alone, that
guarantees this choice at any given budget.

For a nonzero tensor
$\mathcal X\in\mathbb R^{n_1\times\cdots\times n_d}$, recall
the normalised mode-$k$ tail energy
\begin{equation*}
  L_k(r_k)^2
  =
  \frac{
    \sum_{i>r_k}\sigma_i^2(X_{(k)})
  }{
    \|\mathcal X\|_F^2
  }.
\end{equation*}

Throughout this subsection, $\mathbf{r} = (r_1, \ldots, r_d)
\in \mathcal{R}$ denotes the full multilinear rank vector
and $r_k$ its $k$-th component.

The allocator minimises the per-tensor tail-energy surrogate
\begin{equation*}
  f(\mathbf{r})=\sum_{k=1}^d L_k(r_k)^2,
  \qquad
  \mathbf{r}\in\mathcal R
  :=\prod_{k=1}^d\{1,\ldots,n_k\},
\end{equation*}
subject to the Tucker storage constraint
\begin{equation*}
  S(\mathbf{r})=\prod_{k=1}^d r_k
  +\sum_{k=1}^d n_k r_k\le B.
\end{equation*}

We derive a budget-dependent sufficient condition under which
every global minimiser preserves a specified mode at full rank.

Increasing $r_k$ by one requires
\begin{equation*}
  \Delta_k(\mathbf{r})
  =
  n_k+\prod_{j\ne k}r_j
\end{equation*}
additional storage units.  Let
$s(\mathbf{r})=B-S(\mathbf{r})$ denote the remaining budget.
An allocation with $r_k<n_k$ is called \emph{$k$-tight} if
\begin{equation*}
  0\le s(\mathbf{r})<\Delta_k(\mathbf{r}),
\end{equation*}
so that increasing $r_k$ alone would exceed the budget.

For a $k$-tight allocation, consider restoring one rank in
mode $k$ by reducing the rank of another mode $m\ne k$.  The
minimum number of mode-$m$ ranks required to fund this
increment is
\begin{equation}
  p_{k,m}(\mathbf{r})
  =
  \left\lceil
    \frac{\Delta_k(\mathbf{r})-s(\mathbf{r})}
    {n_m+(r_k+1)\prod_{j\ne k,m}r_j}
  \right\rceil.
  \label{eq:exchange-p}
\end{equation}

The corresponding increase in the donor-mode tail energy is
\begin{equation}
  c_{k,m}(\mathbf{r})
  =
  \frac{1}{\|\mathcal X\|_F^2}
  \sum_{i=r_m-p_{k,m}(\mathbf{r})+1}^{r_m}
    \sigma_i^2(X_{(m)}).
  \label{eq:exchange}
\end{equation}
We set $c_{k,m}(\mathbf{r})=\infty$ if
$p_{k,m}(\mathbf{r})\ge r_m$, since the exchange would
violate the minimum rank constraint.

Define
\begin{equation}
  \Gamma_k(B)
  =
  \max_{\substack{
    \mathbf{r}\in\mathcal R,\
    S(\mathbf{r})\le B\\
    r_k<n_k,\
    0\le s(\mathbf{r})<\Delta_k(\mathbf{r})
  }}
  \min_{m\ne k}c_{k,m}(\mathbf{r}).
  \label{eq:exchange-certificate}
\end{equation}
We set $\Gamma_k(B)=0$ if the maximisation set is empty.

Thus, $\Gamma_k(B)$ is the worst-case cost of the cheapest
feasible single-mode exchange over all $k$-tight allocations.
It depends only on the mode spectra, tensor dimensions, and
budget, and can be evaluated without knowing the optimal
allocation. The certificate
compares this worst-case exchange cost to the tail energy of
the mode to be preserved.

\begin{theorem}[Spectral certificate for full-rank preservation]
\label{thm:pinning}

Let $\mathcal X\in\mathbb R^{n_1\times\cdots\times n_d}$
be nonzero, with $d\ge2$, and suppose the feasible set
under budget $B$ is nonempty.

Let $k \in \{1,\ldots,d\}$ with $n_k\ge2$.  If
\begin{equation}
  L_k(n_k-1)^2>\Gamma_k(B),
  \label{eq:pinning-condition}
\end{equation}
then every global minimiser $\mathbf{r}^\star$ of
\begin{equation*}
  \min_{\mathbf{r}\in\mathcal R} f(\mathbf{r})
  \quad\text{subject to}\quad S(\mathbf{r})\le B
\end{equation*}
satisfies $r_k^\star=n_k$.

The condition is budget-dependent and cannot hold when
$\Gamma_k(B)=\infty$.
\end{theorem}

\begin{proof}
Fix a mode $k$ satisfying~\eqref{eq:pinning-condition}.
Suppose, for contradiction, that there exists a global
minimizer $\mathbf{r}^\star$ with $r_k^\star < n_k$.

Since the singular values are non increasing, increasing
$r_k^\star$ by one decreases the objective by at least
\begin{equation*}
  \frac{
    \sigma_{r_k^\star+1}^2(X_{(k)})
  }{
    \|\mathcal X\|_F^2
  }
  \ge L_k(n_k-1)^2>0.
\end{equation*}

If $s(\mathbf{r}^\star)\ge\Delta_k(\mathbf{r}^\star)$, this
increment is feasible without modifying any other rank and
strictly decreases $f$, contradicting optimality.

Hence $\mathbf{r}^\star$ must be $k$-tight.  By the definition
of $\Gamma_k(B)$ and the hypothesis,
\begin{equation*}
  \begin{aligned}
    \min_{m\ne k}c_{k,m}(\mathbf{r}^\star)
    &\le \Gamma_k(B) < L_k(n_k-1)^2
  \end{aligned}
\end{equation*}

Choose a minimising partner $m$ and let
$p=p_{k,m}(\mathbf{r}^\star)$.  The exchange cost is finite,
so $p<r_m^\star$.  Define
\begin{equation*}
  r'_k=r_k^\star+1,\qquad
  r'_m=r_m^\star-p,\qquad
  r'_j=r_j^\star\quad(j\notin\{k,m\}).
\end{equation*}

By the definition of $p$,
\begin{align*}
  S(\mathbf{r}')-S(\mathbf{r}^\star)
  &=
  \Delta_k(\mathbf{r}^\star)
  -
  p\left(
    n_m+(r_k^\star+1)
    \prod_{j\ne k,m}r_j^\star
  \right)\\
  &\le s(\mathbf{r}^\star).
\end{align*}
Therefore $S(\mathbf{r}')\le B$, and the exchange is feasible.

Its objective change is
\begin{align*}
  f(\mathbf{r}')-f(\mathbf{r}^\star)
  &=
  c_{k,m}(\mathbf{r}^\star)
  -
  \frac{
    \sigma_{r_k^\star+1}^2(X_{(k)})
  }{
    \|\mathcal X\|_F^2
  }\\
  &\le
  \Gamma_k(B)-L_k(n_k-1)^2\\
  &<0.
\end{align*}

This contradicts the global optimality of $\mathbf{r}^\star$.
Consequently, every global minimiser satisfies
$r_k^\star=n_k$.
\end{proof}

When the certificate preserves all but one mode at full rank,
the gap between the tail-energy surrogate and the true
reconstruction error can be controlled.

\begin{corollary}[Reconstruction-error tightness under certified
rank preservation]
\label{cor:tightness}

Let $\mathbf{r}^\star$ be a global minimiser of the allocation
problem, and let $\mathcal C$ be a set of modes satisfying the
condition of Theorem~\ref{thm:pinning}.

Let $\widehat{\mathcal X}$ be a truncated HOSVD or ST-HOSVD
approximation at multilinear rank $\mathbf{r}^\star$, with
relative reconstruction error
\begin{equation*}
  e=
  \frac{
    \|\mathcal X-\widehat{\mathcal X}\|_F
  }{
    \|\mathcal X\|_F
  }.
\end{equation*}

For any mode $k\notin\mathcal C$ with $L_k(r_k^\star)>0$,
define
\begin{equation*}
  \delta_k
  =
  \frac{
    \sum_{\substack{
      j\notin\mathcal C\\j\ne k
    }}L_j(r_j^\star)^2
  }{
    L_k(r_k^\star)^2
  }.
\end{equation*}

Then
\begin{equation}
  L_k(r_k^\star)
  \le e
  \le L_k(r_k^\star)\sqrt{1+\delta_k}.
  \label{eq:certified-tightness}
\end{equation}

In particular, if all modes other than $k$ are certified
to retain full rank, then
\begin{equation}
  e^2=f(\mathbf{r}^\star)=L_k(r_k^\star)^2.
  \label{eq:surrogate-exactness}
\end{equation}

The same conclusions hold for a rank-preserving,
non-error-increasing refinement initialised from either
HOSVD variant.

\end{corollary}

\begin{proof}
By Theorem~\ref{thm:pinning}, every certified mode
$j\in\mathcal C$ retains full rank, and hence
\begin{equation*}
  L_j(r_j^\star)=0
  \qquad (j\in\mathcal C).
\end{equation*}

The Eckart--Young--Mirsky
theorem~\citep{eckart1936approximation,mirsky1960symmetric}
applied to each mode unfolding gives the lower bound, and the
standard reconstruction-error bounds for truncated
HOSVD~\citep[Property~10]{delathauwer2000hosvd} and
ST-HOSVD~\citep[Theorem~6.5]{vannieuwenhoven2012sthosvd} give
the upper bound in
\begin{equation*}
  \max_j L_j(r_j^\star)^2
  \le e^2
  \le \sum_j L_j(r_j^\star)^2.
\end{equation*}

Since the certified modes contribute no tail energy,
\begin{equation*}
  L_k(r_k^\star)^2
  \le e^2
  \le
  L_k(r_k^\star)^2
  +
  \sum_{\substack{
    j\notin\mathcal C\\j\ne k
  }}L_j(r_j^\star)^2.
\end{equation*}

Factoring out $L_k(r_k^\star)^2$ and taking square roots
establishes~\eqref{eq:certified-tightness}.

If every mode other than $k$ is certified, then
$\delta_k=0$, so the lower and upper bounds coincide.
Moreover, all other terms in the surrogate vanish, giving
\begin{equation*}
  e^2=L_k(r_k^\star)^2=f(\mathbf{r}^\star).
\end{equation*}

Finally, consider a refinement (such as HOOI) initialised
from either HOSVD variant that preserves all mode ranks and
does not increase the reconstruction error.  Since the mode
ranks are unchanged, the lower bound $L_k(r_k^\star) \le e$
still holds.  Since the refinement can only decrease the
error from its initialisation, the HOSVD upper bound still
applies.  The
bounds defined by ~\eqref{eq:certified-tightness} therefore carries
over.
\end{proof}

\paragraph{Numerical verification}

We evaluate the certificate for Mistral's head mode using
$p_{k,m}$ from~\eqref{eq:exchange-p} and $c_{k,m}$
from~\eqref{eq:exchange}.  Across four prompt draws, five
layers, keys and values, and target ratios of
$2\times$--$5\times$, all 160 configurations satisfy
$L_1(n_h-1)^2 > \Gamma_1(B)$, so the certificate guarantees
full head rank at every tested budget.  The ratio
$L_1(n_h-1)^2/\Gamma_1(B)$ ranges from $3.49$ to $342.33$,
with a median of $13.44$.

Table~\ref{tab:bound-vs-measured} confirms the tightness
predicted by Corollary~\ref{cor:tightness}: the ratio
$E_{\mathrm{group}}/L_k$ stays between $1.00$ and $1.03$
in every row, consistent with the
bound~\eqref{eq:certified-tightness}.

The certificate also holds on all 224 head-mode
configurations of LLaMA-2-13B.  On the grouped-layer mode
($L_g=4$), it holds on 240 of 256 Mistral cells and 305 of
320 LLaMA cells.  The uncertified cells are the value tensor
of the first group of each model, yet the allocator keeps
$r_L = L_g$ there as well.  The condition of
Theorem~\ref{thm:pinning} is therefore sufficient but not
necessary.

\section{Structural Properties}
\label{sec:structural}

The format comparison of Section~\ref{sec:formats} establishes that
Tucker decomposition achieves the lowest reconstruction error at every
compression ratio. Two further properties of the cache, independent of
the decomposition format, affect its low-rank approximability: the
asymmetry between keys and values, and the interaction between
compression and rotary position encoding.

\subsection{The K/V asymmetry}
\label{sec:kv}

Every experiment so far shows the same pattern: at matched compression
ratio, values incur consistently higher reconstruction error than
keys. The spectral analysis of Section~\ref{sec:spectral} shows that
value-token unfoldings require roughly twice the rank of key-token
unfoldings for the same error threshold. The format comparison of
Section~\ref{sec:formats} confirms that the gap is not an artifact of
any single decomposition. Pooled over all formats, layers, and prompt
draws, the median ratio of value error to key error is approximately
$2.4$ and lies in the range $[1.7,\, 3.2]$ across all model--format--ratio
combinations. The asymmetry is a property of the cache spectrum, not of
the algorithm used to approximate it.

A natural explanation follows from the different roles of the two
weight matrices introduced in Section~\ref{sec:background}. The
matrix $W_V$ produces the value cache that are linearly combined
to form the attention output. To preserve content under arbitrary
weightings, $W_V$ must spread energy across many feature directions,
which leads to a flat singular-value spectrum and poor low-rank
approximability. The matrix $W_K$ produces the key cache whose
sole purpose is to determine softmax scores via inner products with
queries(Section~\ref{sec:background}). Concentrating energy into a small number of discriminative
directions suffices for this purpose, which leads to rapid spectral
decay and good low-rank approximability. The singular values of the
learned weight matrices confirm this account: on both models, the
spectrum of $W_K$ drops steeply from its leading singular value
($\sigma_1/\sigma_{\min} \approx 36$ on Mistral), while the spectrum
of $W_V$ remains comparatively flat ($\sigma_1/\sigma_{\min} \approx 6$).
Correspondingly, $W_V$ requires higher rank than $W_K$ to capture
$90\%$ of its energy on nearly every layer.

The K/V asymmetry is a spectral property of the value tensor: values
occupy the same compressible token and feature axes as keys, but with
a flatter spectrum. Any rank allocator must therefore choose
independent ranks for keys and values rather than compressing both at
the same rank.

\section{Rotary Position Embedding and the Pre/Post Penalty}
\label{sec:rope}

The RoPE operator $R$ defined in Section~\ref{sec:background} rotates
each key and query vector by a position-dependent angle. Because $R$
applies an independent orthogonal rotation $R_t$ at each token
position, the full operator is orthogonal on the vectorized tensor
and preserves the Frobenius norm of any error:
$\|R(\mathcal{K}) - R(\tilde{\mathcal{K}})\|_F
= \|\mathcal{K} - \tilde{\mathcal{K}}\|_F$. Rotation therefore does not
change the distance between a tensor and a fixed approximation. It does,
however, change which approximations are efficient at a given rank.
Because $R$ mixes coordinates within each pair at a position-dependent
angle, it can flatten the token-mode and feature-mode spectra and reduce
the tensor's low-rank approximability. Keys can therefore be compressed
before the rotation (pre-RoPE) or after (post-RoPE), and the two
settings need not yield the same error at matched storage.

\begin{table}[htbp]
  \centering
  \caption{Cross-architecture RoPE penalty on keys, Mistral-7B-v0.3 and LLaMA-2-13B, $T=1024$; gap is the relative increase in pooled key error from pre to post.}
  \label{tab:rope-cross}
  \begin{tabular}{lrrrrrr}
    \toprule
    Ratio & Mistral pre & Mistral post & gap\% & LLaMA pre & LLaMA post & gap\% \\
    \midrule
    $2\times$ & $0.1307$ & $0.2117$ & $+62\%$ & $0.1413$ & $0.2311$ & $+64\%$ \\
    $3\times$ & $0.1763$ & $0.2687$ & $+52\%$ & $0.1922$ & $0.2906$ & $+51\%$ \\
    $4\times$ & $0.1982$ & $0.2957$ & $+49\%$ & $0.2171$ & $0.3182$ & $+47\%$ \\
    $6\times$ & $0.2257$ & $0.3335$ & $+48\%$ & $0.2403$ & $0.3440$ & $+43\%$ \\
    $8\times$ & $0.2433$ & $0.3626$ & $+49\%$ & $0.2545$ & $0.3598$ & $+41\%$ \\
    \bottomrule
  \end{tabular}
\end{table}

Table~\ref{tab:rope-cross} compares pre- and post-RoPE key errors on
both models across five compression ratios. Post-RoPE error is
consistently higher, with relative gaps ranging from $41\%$ to $64\%$.
A frozen-allocation control, in which the post-RoPE pass reuses the
pre-RoPE ranks verbatim, produces even larger gaps. The allocator
therefore compensates for the spectral flattening rather than causing it,
and the penalty is attributable to the rotation alone.

\begin{figure}[htbp]
  \centering
  \includegraphics[width=0.85\textwidth]{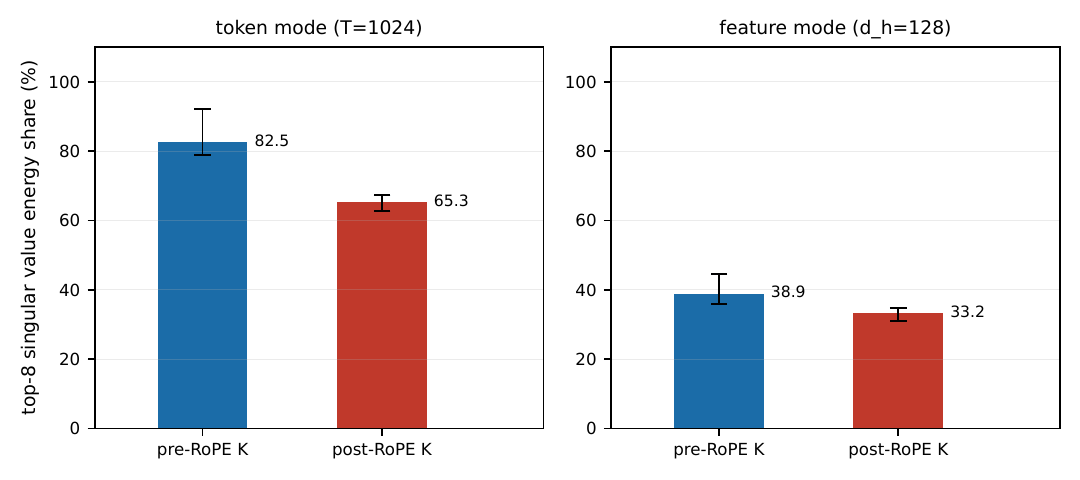}
  \caption{Top-$8$ singular-value energy shares of the token-mode and
  feature-mode unfoldings of the key tensor on Mistral-7B-v0.3, before
  and after RoPE. Bars show the mean over all $8$ layer groups
  ($L_g = 4$). Error bars show the range across groups.}
  \label{fig:rope-energy}
\end{figure}

Figure~\ref{fig:rope-energy} shows where the spectral flattening
occurs, averaged over all layer groups of Mistral. The top-$8$
token-mode energy share drops from $82.5\%$ to $65.3\%$, and the
top-$8$ feature-mode share drops from $38.9\%$ to $33.2\%$. Both
drops are consistent across groups, as the error bars confirm. The
token-mode effect is larger because RoPE mixes coordinates at
position-dependent angles that vary across the full sequence length,
spreading energy away from the leading token-mode components.
Post-RoPE, the key feature spectrum flattens, and the compression advantage that keys hold over values in
the feature mode shrinks.

\section{Conclusions}
\label{sec:conclusion}

The KV cache divides into two spectral classes.
The token and feature modes carry low-rank structure that admits rank reduction,
whereas the head and grouped-layer modes are index-like
(Definition~\ref{def:index-like}).  Theorem~\ref{thm:pinning} certifies
this partition from the mode spectra alone, and holds on every tested head-mode
configuration of both models at target ratios $2{\times}$--$5{\times}$.  Among the
four decompositions, Tucker achieves the lowest reconstruction error at every
ratio, because it can leave the index-like modes untouched and concentrate its
budget on the compressible ones.

Two further spectral properties affect the compressible modes
without touching the index-like ones.  Values reach a higher
reconstruction-error floor than keys at matched storage, an asymmetry traced to
the flatter spectrum of the value projection~$W_V$.  Post-RoPE keys lose
$41\%$--$64\%$ of their pre-RoPE compressibility on both models, because the
position-dependent rotation flattens the token- and feature-mode spectra.
Keys and values therefore need independent rank budgets, and keys
should be compressed before the rotary embedding.

All reconstruction sweeps use a single sequence length ($T=1024$),
and the spectral analysis covers two models.  CP could not be fitted on the
LLaMA tensor within the tested hardware.  Whether the token-mode
structure measured on a filled prompt persists during incremental generation
is an open question.  The companion paper
JoLT~\citep{krishnan2026jolt} operationalizes the constraints above into
a partial-Tucker compressor with a rotated low-bit residual.

\section*{Reproducibility Statement}

The authors have followed reproducibility principles valued by the scientific computing community.  Code and data that allow readers to reproduce the results in this paper are available at \url{https://github.com/rahulk98/JoLT-Master-Thesis}.

\section*{AI Use Statement}

The authors used generative AI tools in the following way:
LLM-based coding assistants were used to assist with software implementation; generative AI tools were used to assist with manuscript drafting, and AI-assisted search was used to retrieve related work and to check citation records against publisher metadata.  All AI-assisted code and text were reviewed and verified by the authors. The authors take full responsibility for the final content of the paper.

\bibliographystyle{unsrtnat}
\bibliography{references}

\end{document}